\documentclass[11pt,leqno]{amsart}
\usepackage{amscd}
\usepackage{color}
\usepackage[symbol]{footmisc}
\usepackage{amssymb}
\usepackage{amsfonts}
\usepackage{tikz-cd}
\usepackage{latexsym}
\usepackage{verbatim}
 \usepackage[backref=page]{hyperref}
 \renewcommand*{\backref}[1]{}
\renewcommand*{\backrefalt}[4]{%
 	\ifcase #1 (Not cited).%
	\or        (Cited on page~#2).%
	\else      (Cited on pages~#2).%
	\fi}
\hypersetup{colorlinks=true,linkcolor=black,citecolor=black,urlcolor=black}

\newcommand{\Z}{\mathbb{Z}}

\newcommand{\R}{\mathbb{R}}
\newcommand{\C}{\mathbb{C}}

\renewcommand{\epsilon}{\varepsilon}
\renewcommand{\phi}{\varphi}
\renewcommand{\Re}{\mathrm{Re}}

\theoremstyle{plain}
\newtheorem{teor}{Theorem}[section]
\newtheorem{conjecture}[teor]{Conjecture}
\newtheorem{prop}[teor]{Proposition}
\newtheorem{lem}[teor]{Lemma}

\newtheorem{rmk}[teor]{Remark}
\theoremstyle{definition}

\newtheorem{es}[teor]{Example}
\title[HKT geometry on LCHK manifolds]{HKT geometry on locally conformally hyperk\"ahler manifolds}
\begin{document}

\thanks{This work was supported by GNSAGA of INdAM. The first-named author was supported by the DFF Sapere Aude grant \lq \lq Conformal geometry: metrics and cohomology''}

\address{Giovanni Gentili \\ Aarhus University\\ Institut for Matematik\\
8000, Aarhus C\\ Denmark.}
\email{giovanni.gentili@math.au.dk}

\medskip

\address{Luigi Vezzoni \\
Dipartimento di Matematica ``G. Peano'', Universit\`{a} degli studi di Torino \\
Via Carlo Alberto 10, Torino, Italy}

\email{luigi.vezzoni@unito.it}

\author{Giovanni Gentili and Luigi Vezzoni}


\begin{abstract}
We study the HKT geometry of compact locally conformally hyperk\"ahler manifolds. We prove that on these manifolds the quaternionic Monge--Amp\`ere equation is always solvable, extending results of Alesker and of Dinew--Sroka, and we confirm the conjecture on the maximal existence time of the HKT--Ricci flow. On quaternionic Hopf surfaces we construct non-trivial HKT solitons and prove convergence of the normalized flow for symmetric initial data.
\end{abstract}

\maketitle

\section{Introduction}
In \cite{AV10} Alesker and Verbitsky introduced a Monge--Amp\`ere-type equation on HKT manifolds as a natural adaptation of the classical complex Monge--Amp\`ere equation to the hypercomplex case. Given a compact HKT manifold $(M^{4n},I,J,K,g)$ with HKT form $\Omega$, the {\em quaternionic Monge--Amp\`ere equation} consists of 
\begin{equation}\label{CY}
(\Omega+\partial\partial_J\phi)^n={\rm e}^{F+b}\Omega^n\,,\quad \Omega+\partial\partial_J\phi>0\,,\quad 
{\rm sup}_M \varphi=0\,,
\end{equation}
where $F$ is a given smooth map on $M$. In \cite[Conjecture 1.5]{AV10} it was conjectured that the equation can always be solved at least when all three complex structures $(I,J,K)$ have holomorphically trivial canonical bundle. The solvability of equation \eqref{CY} in its full generality is still open, but several partial results have been proved in the literature (see e.g. \cite{Ale13,AS13,AS17,AV10,BGV22,BGV23,DS23,GV22,GV23,Sro20,Sro24}, and the references therein). The most significant advances so far were obtained by Alesker in \cite{Ale13} and by Dinew and Sroka in \cite{DS23}. In \cite{Ale13} the solvability of the equation is shown on flat hyperk\"ahler manifolds, while in \cite{DS23} the flatness assumption was removed. 

\medskip 
In our first main result we extend the result of Dinew and Sroka to locally conformally hyperk\"ahler manifolds (LCHK manifolds):

\begin{teor}\label{main1}
Let $(M,I,J,K,g)$ be a compact HKT manifold. If $(I,J,K)$ admits a compatible LCHK metric $\tilde g$, then equation \eqref{CY} can be solved for every $F\in C^{\infty}(M)$.  
\end{teor}

We remark that in view of \cite{OPS03} any compact LCHK manifold $(M,I,J,K,\tilde g)$ has a canonical HKT metric $g$ which turns out to satisfy the HKT--Einstein equation introduced in \cite{FG26} 
\begin{equation}\label{Einsteinlambda}
\mathrm{Ric}_I(g)^+=c\,g\,,
\end{equation}
where $c\in \R$, ${\rm Ric}_I(g)$ is the first Chern--Ricci tensor of $(g,I)$ and given a symmetric $2$-tensor $\psi$ on $M$, we denote by 
$\psi^{+}:=\frac12(\psi+J\psi)$ the $J$-invariant component of $\psi$. 

\medskip 
In \cite{BGV22,Zha22} a parabolic approach was proposed to study equation \eqref{CY}. This led the authors to consider the geometric flow 
\begin{equation}\label{flow}
\partial_tg_t=-{\rm Ric}_I(g_t)^{+}
\end{equation}
introduced in \cite{BGV23} on a compact hypercomplex manifold $(M,I,J,K)$. Given a hyperhermitian metric $g$ on $(M,I,J,K)$, the flow is well-posed and, in analogy to the classical K\"ahler--Ricci flow, can be reduced to a parabolic Monge--Amp\`ere-type equation (see \cite{BGV23}). The flow was also studied by Brienza and Streets in \cite{BS26}, where they showed that, when the metric is HKT, it belongs to the family of Hermitian curvature flows introduced by Streets and Tian in \cite{ST11}. They also proved a monotonicity formula for the Chern scalar curvature along the flow and investigated its behaviour on complex surfaces with a hypercomplex structure.

\medskip 
We recall the following 

\begin{conjecture}[\cite{BGV23,BS26}]\label{conj}
On a compact HKT manifold $(M,I,J,K,g_0)$ the maximal-time solution to the flow \eqref{flow} with initial condition $g_0$ is defined for $t\in[0,T)$, where  
\[
T=
\sup\{
t\geq 0\,\,:\,\, [\omega_0]_{\rm HKT}-tc_1^{\rm HKT}(M)>0\}\,.
\]
\end{conjecture}
In contrast to \cite{BS26}, we formulate the conjecture in terms of the fundamental \((1,1)\)-form \(\omega_0\), rather than the HKT \((2,0)\)-form \(\Omega_0\). Here we have adopted the following notation: given a $J$-anti-invariant form $\alpha\in \Lambda^{1,1}_I$, $[\alpha]_{\rm HKT}$ denotes 
$$
[\alpha]_{\rm HKT}=\{\alpha+(i\partial \bar\partial\psi)^{-}\,\,:\,\,\psi \in C^{\infty}(M)\}\,,
$$
and 
$$
c_1^{\rm HKT}(M)=[{\rm Ric}_I(\omega)^{-}]_{\rm HKT}\,,
$$
where ${\rm Ric}_I(\omega)$ is the Chern--Ricci form of some hyperhermitian metric on $(M,I,J,K)$ and given a $2$-form $\beta$ on $M$, $\beta^{-}:=\frac12(\beta-J\beta)$. The class $c_1^{\rm HKT}(M)$ does not depend on the choice of the hyperhermitian metric (see Section \ref{preliminaries}). 

So far Conjecture \ref{conj} was confirmed on hyperk\"ahler manifolds \cite[Theorem 1.1]{BGV23} and on complex surfaces with a hypercomplex structure \cite{BS26}. We confirm the conjecture on LCHK manifolds:
\begin{teor}\label{main2}
Conjecture $\ref{conj}$ holds on compact HKT manifolds admitting a compatible LCHK metric. 
\end{teor}

The next natural problem is the study of the asymptotic behaviour of the flow for $t\to T$. As pointed out in \cite{BS26}, this is a  delicate problem. Let $(M,I,J,K,g)$ be a compact HKT--Einstein manifold and let $g_0$ be an HKT metric on $(M,I,J,K)$ whose $I$-fundamental form $\omega_0$ belongs to  $[\omega]_{\rm HKT}$. It is quite natural to wonder if the normalized flow
\begin{equation}\label{nflow}
\partial_tg_t=-{\rm Ric}_I(g_t)^{+}+c\,g_t\,,\quad g_{|t=0}=g_0\,,
\end{equation}
converges to $g$ whenever Conjecture \eqref{conj} holds. This happens, for instance, when $g$ is hyperk\"ahler and $c=0$ \cite{BGV23}. As pointed out in \cite{BS26}, things work differently on the quaternionic Hopf surfaces $H_\lambda$, where there are non-equivalent HKT--Einstein metrics in the fixed HKT class of the canonical HKT--Einstein metric $g_H$. 
      
\begin{teor}\label{main3}
For $\lambda>{\rm e}^{\pi}$, the quaternionic Hopf surface $H_\lambda$ admits a hyperhermitian metric $h$ satisfying the soliton equation 
\begin{equation}
{\rm Ric}_I(h)^{+}=h-\mathcal{L}_{V}(h)\,,\quad  \mathcal{L}_{V}(h)\neq 0\,, 
\end{equation}
where $V=-\frac12 \theta^{\sharp}$, and $\theta$ is the Lee form of $g_H$. 
Moreover, if $g_t$ is a long-time solution of \eqref{nflow} on $H_\lambda$ such that 
$$
c=\frac{2\,\mathrm{Vol}(g_H)}{\displaystyle\int_{H_\lambda}\mathrm{tr}_{g_H}g_0\,d\mu_{g_H}}\,,\quad  \mathcal{L}_V(g_0)=0\,,
$$ 
and satisfying one of the following two conditions
\begin{enumerate}
\item[1.] $C^{-1}g_H\leq g_t\leq C g_H$ for every $t\in [0,\infty)$ and a uniform constant $C$;

\vspace{0.1cm}
\item[2.] $\mathcal{L}_{IV}g_0=0$;
\end{enumerate}
then $g_t$ converges smoothly to an HKT--Einstein metric as $t\to \infty$. 
\end{teor}

Next we describe the structure of the proofs:\\
Following Yau's proof of the classical Calabi conjecture \cite{Yau78}, the natural approach to study equation \eqref{CY} is via a priori estimates on the solutions. The $C^0$ estimate is completely understood  \cite{AS17,AV10,Sro20,Sro24}, and the only missing estimate is an upper bound on the Chern Laplacian of the solution (see e.g. \cite[Theorem 3.1]{BGV23} and \cite[Proposition 7]{GV26}). In order to show this bound in the setting of Theorem \ref{main1}, we consider the test function 
\[
Q:=\log \mathrm{tr}_{\tilde g}g_\phi -\log(1+\|\varphi\|_{C^0}+\phi)\,,
\]
where $g_\phi$ is the HKT metric corresponding to $\Omega+ \partial \partial_J \phi$. The result is then obtained by manipulating the inequality $\Delta_\varphi Q\leq 0$ at a maximum point $x_0$ of $Q$ with the aid of the critical equation $\partial Q(x_0)=0$. An important technical tool in our computation is an adaptation of a proposition in Alesker's paper \cite{Ale13} to our case. Our argument does not adapt to the general case without the LCHK assumption (see Remark \ref{rem}). The proof of Theorem \ref{main2} uses a similar argument. 

The construction of the soliton in Theorem \ref{main3} is obtained by reducing the problem to an ODE on the function $u$ which determines the metric by ${\rm e}^{u(\log |q|)}g_H$. For the second part of the statement, the condition $\mathcal{L}_{V}g_0=0$ allows us to regard the flow as a flow on the $3$-sphere $S^3$, while the additional assumption $\mathcal{L}_{V}g_0=\mathcal{L}_{IV}g_0=0$ reduces the flow to one on $S^2$. In this last case the flow coincides with one considered by Castéras in \cite{Cas15}.

\bigskip 
\noindent {\bf Notation.} In this paper, we adopt the following notation: all tensor components are expressed with respect to complex coordinates that are holomorphic with respect to the first complex structure~$I$. We use a comma in the indices to denote partial derivatives. When differentiating functions, the comma is omitted.

\bigskip 
\noindent {\bf Acknowledgments.} The authors are grateful to Beatrice Brienza and Jeffrey Streets for useful discussions and their interest in this work.

\section{Preliminaries}\label{preliminaries}

A hypercomplex manifold is a smooth manifold of real dimension $4n$ equipped with a triple of complex structures $I,J,K$ satisfying the quaternionic relations
\[
IJ=-JI=K\,. 
\]
A hyperhermitian metric on a hypercomplex manifold $(M,I,J,K)$ is a Riemannian metric $g$ which is compatible with each complex structure, i.e., 
$$
g(I\cdot,I\cdot)=g(J\cdot,J\cdot)=g(K\cdot,K\cdot)\,. 
$$
The metric $g$ induces the three fundamental forms 
$$
\omega(\cdot,\cdot):=g(I\cdot ,\cdot)\,,\quad \omega_J(\cdot,\cdot):=g(J\cdot ,\cdot)\,,\quad \omega_K(\cdot,\cdot):=g(K\cdot ,\cdot)
$$
and the form  
$$
\Omega:=\omega_J+i\omega_K\,.
$$
The form $\Omega$ is of type $(2,0)$ with respect to $I$ and satisfies the {\em q-reality} condition 
\begin{equation}\label{q-real}
\Omega(J\cdot ,J\cdot)=\overline{\Omega(\cdot,\cdot)}\,,
\end{equation}
as well as the \emph{q-positivity} condition 
\begin{equation}\label{positive}
\Omega(Z,J\bar Z)>0
\end{equation}
for every non-zero vector field $Z$ of type $(1,0)$ with respect to $I$. Moreover, there is a one-to-one correspondence 
between hyperhermitian metrics and $(2,0)$-forms satisfying \eqref{q-real} and \eqref{positive} since 
every $\Omega\in \Lambda^{2,0}_I$  satisfying \eqref{q-real} and \eqref{positive} induces a hyperhermitian metric $g$ via the relation 
\[
g=\Re (\Omega(\cdot,J\cdot))\,.
\]
 
Alternatively, a hyperhermitian metric can be given in terms of its $I$-fundamental form, since the Hermitian metric induced by a positive real form in $\Lambda^{1,1}_I$ which satisfies the $J$-anti-invariant relation
$$
J\omega=-\omega
$$
is hyperhermitian.




A hyperhermitian metric is hyperk\"ahler if and only if the HKT form is closed. In this case the three Bismut connections induced by $(g,I)$, $(g,J)$ and $(g,K)$ agree with the Levi-Civita connection of $g$. In \cite{HP96} Howe and Papadopoulos generalized the notion of hyperk\"ahler manifold to geometry with torsion by introducing {\em HKT manifolds}. According to the definition given in \cite{HP96}, a hyperhermitian manifold $(M,I,J,K,g)$ is HKT (hyperk\"ahler with torsion), if the  three Bismut connections induced by $(g,I)$, $(g,J)$ and $(g,K)$ coincide. Later it was proved in \cite{GP00} that the HKT condition can be characterized in terms of the HKT form $\Omega$ as 
$$
\partial \Omega=0\,,
$$
where $\partial$ is with respect to $I$.

Let $(M,I,J,K,g)$ be an HKT manifold and consider the operator $\partial_J:= J^{-1}\bar\partial J$. Since $J$ anti-commutes with $I$, we have $\partial_J\colon \Lambda^{r,0}_I\to \Lambda^{r+1,0}_I$ and if $\varphi$ is a real smooth function, then 
$\partial \partial_J\varphi\in\Lambda^{2,0}_I$.  Moreover $\varphi$ is {\em $\Omega$-plurisubharmonic} if $\Omega+\partial \partial_J\phi$ is q-positive. In this case we write 
$$
\Omega+\partial \partial_J\varphi>0
$$
and we set $\Omega_\varphi:=\Omega+\partial \partial_J\varphi$. Note that any {\em $\Omega$-plurisubharmonic} function induces a hyperhermitian metric which we denote by $g_\varphi$ and the $I$-fundamental form $\omega_\phi$ of $g_{\varphi}$ is given by 
$$
\omega_{\varphi}=\omega+(i\partial \bar\partial \varphi)^{-}\,.
$$
This naturally leads us to consider the classes $[\omega]_{\rm HKT}$ and $c_1^{\rm HKT}(M)$ given in the Introduction.
Note that $c_1^{\rm HKT}(M)$ does not depend on the choice of the hyperhermitian metric. Indeed, if $\omega'$ is another hyperhermitian metric, then $\mathrm{Ric}_I(\omega')-\mathrm{Ric}_I(\omega)=-i\partial\bar\partial\log(\omega'^{2n}/\omega^{2n})$, so that $[\mathrm{Ric}_I(\omega')^-]_{\rm HKT}=[\mathrm{Ric}_I(\omega)^-]_{\rm HKT}$.

It follows directly from the definition \eqref{Einsteinlambda} that the existence of a HKT--Einstein metric with non-zero Einstein constant forces $c_1^{\rm HKT}(M)$ to have a sign. In the compact case, it was shown in \cite[\S 6.2]{FG26} that $c_1^{\rm HKT}(M)<0$ is not allowed. Other consequences and obstructions have been investigated in \cite{BFGV26,GL25}. Compact positive HKT--Einstein manifolds have been found in \cite{BM,FG26}.

\medskip 
Next we focus on LCHK manifolds. These are hypercomplex manifolds $(M,I,J,K)$ admitting a hyperhermitian metric $g$ which is locally conformal to a hyperk\"ahler one, i.e. any point $p$ of $M$ has an open neighborhood $U$ such that 
$$
g={\rm e}^f \,\hat g\quad \mbox{ on }U\,,
$$
where $f\colon U\to \R$ is a smooth function and $\hat g$ is a hyperk\"ahler metric on $(U,I,J,K)$. Equivalently, there exists a closed $1$-form $\theta$ such that $d\Omega=\theta \wedge \Omega$ (see \cite{OO19}). Examples of compact non-K\"ahler LCHK manifolds are given by {\em quaternionic Hopf manifolds} defined 
as
\[
H_\lambda^n
:=
\bigl(\mathbb H^n\setminus\{0\}\bigr)/
\sim \,,\quad q\sim \lambda^k q\,,\quad k\in\Z\,,
\]
where $\lambda>1$. A full classification of compact homogeneous LCHK manifolds is given in \cite{OP97}. We also mention \cite{GOPP06} where LCHK reductions and their relation to HKT geometry are studied. By Verbitsky's structure theorem~\cite[Theorem~11.6]{Ver04}, every compact LCHK manifold which is not
hyperk\"ahler is the quotient of a hyperk\"ahler cone over a compact $3$-Sasakian manifold by the cyclic group generated by a dilation composed with a $3$-Sasakian automorphism. More precisely, if  $(M^{4n},I,J,K,g)$ is a compact locally conformally hyperk\"ahler manifold which is not hyperk\"ahler, there exist a compact $3$-Sasakian manifold $(X^{4n-1},g_X)$, a $3$-Sasakian automorphism $\psi\colon X\to X$, and a number $\lambda>1$ such that
\[
M \cong C(X)/\langle\gamma\rangle\,,
\qquad
\gamma(x,r)=(\psi(x),\lambda r)\,,
\]
where
\[
C(X)=X\times(0,\infty)\,,
\qquad
g_C=dr^2+r^2g_X
\]
is the hyperkähler cone over $X$.

The identification preserves the hypercomplex structures and identifies the conformal class $[g]$ with the one of the metric $g_0$ induced by $r^{-2}g_C$. In particular, setting $t=\log r$, this metric is
\[
g_0=dt^2+g_X\,,
\qquad
(x,t)\sim(\psi(x),t+\log\lambda)\,.
\]
Conversely, every such quotient is a compact LCHK manifold which is not hyperk\"ahler. Quaternionic Hopf manifolds correspond to the case $X=S^{4n-1}$ and $\psi={\rm Id}$. 

\medskip 
In view of \cite[Theorem~14]{OPS03}, every compact non-hyperk\"ahler LCHK manifold  $(M^{4n},I,J,K,\tilde g)$ admits
an HKT metric $g$. This metric is defined as follows:\\
Up to conformally rescaling $\tilde g$ we may assume without loss of generality that $(\tilde g,I)$ is {\em Gauduchon}. This means that the $I$-fundamental form $\tilde \omega$ of $\tilde g$ satisfies 
$$
\partial\bar\partial\tilde \omega^{2n-1}=0\,.
$$
The Gauduchon condition combined with the LCHK assumption implies that the Lee form 
$$
\tilde\theta:=-\frac{1}{2n-1}Id^*_{\tilde g}\tilde \omega
$$ 
of $(\tilde g,I)$ is parallel with respect to the Levi-Civita connection of $\tilde g$ \cite{Gau95,OPS03}. Since $\tilde\theta$ is parallel, it has constant length, and up to rescaling $\tilde g$ we may assume $|\tilde\theta|_{\tilde g}=1$. Then $g$ is defined as 
\begin{equation}\label{eq:OPSconstr}
g=n\tilde  g- \frac{n}{2}  \left(\tilde\theta \otimes \tilde\theta+I\tilde\theta \otimes I\tilde\theta+J\tilde\theta \otimes J\tilde\theta+K\tilde\theta \otimes K\tilde\theta\right).
\end{equation}
If we regard $M$ as a quotient of the cone on a $3$-Sasaki manifold, $g$ is the metric with fundamental form $\omega$ induced by the projection of 
$$
2n (i\partial \bar\partial \log(r^2))^{-}\,.
$$
Here the minus denotes the $J$-anti-invariant component with respect to the hypercomplex structure of the cone. 

\begin{prop}
The metric $g$ satisfies the HKT--Einstein equation
\[
g=\mathrm{Ric}_I(g)^+\,.
\]
\end{prop}
\begin{proof}
In \cite{OPS03} it is proved that $g$ is an HKT metric with potential $1$-form $n\tilde \theta$, in other words
\[
\omega=\frac{n}{2}\left(d(I\tilde \theta)+d^c_J(K\tilde \theta)\right)\,.
\]
In our conventions, \cite[Proposition 2.6(b)]{FG26} says
\[
dI\theta=\frac{1}{2n-1}\mathrm{Ric}_I(\omega)\,,
\]
where $\theta$ is the Lee form of $g$. We claim that $\theta=\frac{n}{2n-1}\tilde \theta$. Assuming the claim we deduce
\[
\omega=\frac{2n-1}{2}\left(d(I \theta)+d^c_J(K\theta)\right)
=\frac{1}{2}\left(\mathrm{Ric}_I(\omega)-J\mathrm{Ric}_I(\omega)\right)
=\mathrm{Ric}_I(\omega)^-\,,
\]
as desired, so it only remains to prove that $\theta=\frac{n}{2n-1}\tilde \theta$. From the construction of $g$ it is not difficult to see that
\[
\Omega^{n}=\frac{n^n}{2}\tilde \Omega^{n}\,,
\]
therefore, we deduce $\alpha_{\Omega}=\alpha_{\tilde \Omega}$, where $\alpha_\Omega$ is the $(1,0)$-form such that $\partial \bar \Omega^n=\alpha_\Omega \wedge \bar \Omega^n$ and analogously for $\alpha_{\tilde \Omega}$. On the other hand, since $\Omega$ is HKT $\alpha_{\Omega}=(2n-1)\theta^{1,0}$ (see \cite[Lemma 2.2]{BGV23}), whereas locally $\tilde \Omega=\mathrm{e}^f\hat \Omega$, where $\hat \Omega$ is the $(2,0)$ form associated to a local hyperk\"ahler metric, hence
\[
\alpha_{\tilde \Omega}\wedge \bar{\tilde \Omega}^n
=\partial \bar{\tilde \Omega}^n
=\partial (\mathrm{e}^{nf}\bar{\hat \Omega}^n)
=n\partial f \wedge \mathrm{e}^{nf}\bar{\hat \Omega}^n
=n\partial f \wedge \bar{\tilde \Omega}^n
\]
implying $\alpha_{\tilde \Omega}=n\partial f=n\tilde \theta^{1,0}$.
\end{proof}

\begin{es}\label{Example}
The  metric
$$
\tilde g=\frac{4}{\,|q|^2}g_{\rm Eu}\,,
$$
on  $\mathbb H^n\setminus\{0\}$, where $g_{\rm Eu}$ is the Euclidean metric, induces an LCHK metric on the Hopf manifold $H_\lambda^n$. The induced HKT metric defined by \eqref{eq:OPSconstr} is the projection to $H^n_\lambda$ of the hyperhermitian metric $g_H$ corresponding to the form 
$$
\omega=2n \left(i \partial\bar \partial (\log |q|^2)\right)^{-}
$$
on $\mathbb H^n\setminus\{0\}$.
\end{es}

\section{Proof of Theorem \ref{main1}}
Now we focus on the proof of Theorem \ref{main1}. Let $(M^{4n},I,J,K,g)$ be a compact HKT manifold and fix the datum $F\in C^{\infty}(M)$. By assumption there exists a LCHK metric $\tilde g$ on $(M,I,J,K)$. In order to prove Theorem \ref{main1}, it is enough to show that $\mathrm{tr}_{\tilde g}g_\phi$ has an a priori upper bound \cite[Proposition 7]{GV26}. Furthermore, since the $C^0$ estimate has been established, such upper bound can depend on the $C^0$ norm of the solution. 

\medskip
We assume without loss of generality that $(\tilde  g,I)$ is Gauduchon in order to have the Lee form of $\tilde  \omega$ parallel with respect to the Levi-Civita connection of $\tilde  g$. 

\medskip 
Let $(\varphi,b)$ be a solution to \eqref{CY}.
Consider the test function
\begin{equation} \label{Q}
Q:=\log \mathrm{tr}_{\tilde g}g_\phi -\log(A+\phi)\,,
\end{equation}
where $A=\|\phi\|_{C^0}+1$. Denote $\Delta_\phi u:=g_\phi^{r\bar s}u_{r\bar s}$ the Chern-Laplacian with respect to $g_\phi$. We have
\[
\Delta_\phi Q=
\frac{\Delta_\phi\mathrm{tr}_{\tilde g}g_\phi}
{\mathrm{tr}_{\tilde g}g_\phi}
-\frac{|\partial\mathrm{tr}_{\tilde g}g_\phi|_{g_\phi}^2}
{(\mathrm{tr}_{\tilde g}g_\phi)^2}+\frac{\mathrm{tr}_{g_\phi}g-2n}{A+\phi}
+\frac{|\partial\phi|_{g_\phi}^2}{(A+\phi)^2}\,,
\]
where we have used
\[
\Delta_\phi \phi=\mathrm{tr}_{g_\phi}(i\partial \bar \partial \phi)=\frac{1}{2}\mathrm{tr}_{g_\phi}(i\partial \bar \partial \phi-iJ\partial \bar \partial \phi)=\mathrm{tr}_{g_\phi}g_\phi-\mathrm{tr}_{g_\phi}g=2n-\mathrm{tr}_{g_\phi}g\,.
\]
Assume that $Q$ achieves a maximum at $x_0\in M$. From $\partial Q(x_0)=0$ we infer 
\begin{equation}\label{pQ=0}
\frac{\partial\mathrm{tr}_{\tilde g}g_\phi}
{\mathrm{tr}_{\tilde g}g_\phi}
=\frac{\partial\phi}{A+\phi} \quad \mbox{ at }x_0\,,
\end{equation}
furthermore
\[
0\geq\Delta_\phi Q
=
\frac{\Delta_\phi\mathrm{tr}_{\tilde g}g_\phi}
{\mathrm{tr}_{\tilde g}g_\phi}
+\frac{\mathrm{tr}_{g_\phi}g-2n}{A+\phi} \quad \mbox{ at }x_0\,.
\]
Write $\tilde g={\rm e}^f\hat g$ in an open neighbourhood  $U$ of $x_0$, where  
$\hat g$ is hyperk\"ahler.
Using 
\[
\mathrm{tr}_{\tilde g}g_\phi
={\rm e}^{-f}\mathrm{tr}_{\hat g}g_\phi\,,
\]
we infer 
\begin{equation}\label{A-team}
0\geq
\frac{\Delta_\phi\mathrm{tr}_{\hat g}g_\phi}
{\mathrm{tr}_{\hat g}g_\phi}
-\Delta_\phi f+|\partial f|_{g_\phi}^{2}-\frac{2}{\mathrm{tr}_{\hat g}g_\phi}
\mathrm{Re}\langle\partial f,\partial\mathrm{tr}_{\hat g}g_\phi\rangle_{g_\phi}
+\frac{\mathrm{tr}_{g_\phi}g-2n}{A+\phi} \quad \mbox{ at }x_0\,.
\end{equation}
We have the following

\begin{lem}\label{Gauduchon}
For every hyperhermitian metric $h$ on $(U,I,J,K)$ we have
\[
\Delta_h f+|\partial f|_h^2\leq 0\,.
\]
\end{lem}
\begin{proof}
The conformal change formula for the Levi-Civita connection is
\[
\nabla^{\tilde g}_X Y
=
\nabla^{\hat g}_X Y
+\frac12\bigl(
df(X)Y+df(Y)X
-\hat g(X,Y)\mathrm{grad}_{\hat g}f
\bigr).
\]
Since the Lee form $\tilde \theta=df$ is parallel, we have $\nabla^{\tilde g}df=0$. It follows that
\begin{equation}\label{Hessian}
\mathrm{Hess}_{\hat g}f
=
df\otimes df-\frac12|df|_{\hat g}^2\hat g\,.
\end{equation}
Fix a point and choose a $\hat g$-orthonormal frame $e_i, Ie_i,Je_i, Ke_i$, $i=1,\ldots,n$, adapted to $(I,J,K)$ which diagonalizes $h$ relative to $\hat g$.
By hyperhermitianity, the eigenvalue of $h$ is the same on the four vectors of each block; denote it by $\lambda_i>0$. Set
\[
r_i=
df(e_i)^2+df(Ie_i)^2+df(Je_i)^2+df(Ke_i)^2\,.
\]
Tracing \eqref{Hessian} gives
\[
\begin{split}
\mathrm{tr}_{h}\mathrm{Hess}_{\hat g}f&=\sum_{i=1}^n\frac{1}{\lambda_i}\left(
\mathrm{Hess}_{\hat g}f(e_i,e_i)
+\mathrm{Hess}_{\hat g}f(Ie_i,Ie_i)+
\mathrm{Hess}_{\hat g}f(Je_i,Je_i)
+\mathrm{Hess}_{\hat g}f(Ke_i,Ke_i)\right)
\\
&=\sum_{i=1}^n\frac{r_i-2|df|_{\hat g}^2}{\lambda_i}\,.
\end{split}
\]
The associated $\hat g$-unitary frame of type $(1,0)$ is
\[
Z_{2i-1}=\frac{e_i-\mathrm{i}Ie_i}{\sqrt{2}},
\qquad
Z_{2i}=\frac{Je_i-\mathrm{i}Ke_i}{\sqrt{2}}.
\]
Extend $\mathrm{Hess}_{\hat g}f$ complex-bilinearly to $TU\otimes \C$. Since $\hat g$ is K\"ahler with respect to $I$, for every $Z,W\in T_I^{1,0}U$ we have $ (\partial\bar\partial f)(Z,\bar W) = \mathrm{Hess}_{\hat g}f(Z,\bar W) $. Hence
\[
\Delta_h f =\frac12\sum_{i=1}^n\frac{r_i-2|df|_{\hat g}^2}{\lambda_i}, \qquad
|\partial f|_h^2 =\frac12\sum_{i=1}^n\frac{r_i}{\lambda_i}.
\]
Therefore
\[
\Delta_h f+|\partial f|_h^2 = \sum_{i=1}^n\frac{r_i-|df|_{\hat g}^2}{\lambda_i} = -\sum_{i=1}^n
\frac{\sum_{j\neq i}r_j}{\lambda_i}
\leq 0\,. \qedhere 
\]
\end{proof}

From Lemma \ref{Gauduchon} we have $-\Delta_\phi f-|\partial f|_{g_\phi}^{2}\geq 0$ and \eqref{A-team} reduces to 
\begin{equation}\label{C-team}
0\geq\frac{\Delta_\phi\mathrm{tr}_{\hat g}g_\phi}
{\mathrm{tr}_{\hat g}g_\phi}+2|\partial f|_{g_\phi}^2-2
\mathrm{Re}\langle\partial f,\partial\log \mathrm{tr}_{\hat g}g_\phi\rangle_{g_\phi}
+\frac{\mathrm{tr}_{g_\phi}g-2n}{A+\phi} \quad \mbox{ at }x_0\,.
\end{equation}
The next step consists in showing that 

\begin{equation}\label{B-team}
\frac{\Delta_\phi\mathrm{tr}_{\hat g}g_\phi}{\mathrm{tr}_{\hat g}g_\phi}
\geq \frac12 \left| \partial\log\mathrm{tr}_{\hat g}g_\phi
\right|_{g_\phi}^{2}- \frac{C(1+\mathrm{tr}_{g_\phi}g)}{\mathrm{tr}_{\hat g}g_\phi}\,,
\end{equation}
where $C$ is a positive constant independent of $\varphi$. We need the following technical proposition which is interesting in its own right.

\begin{prop}\label{Alesker}
Let $\hat g,h$ be hyperhermitian metrics on $(U,I,J,K)$, where $h$ is HKT. Take $I$-holomorphic coordinates $(z^1,\ldots,z^{2n})$ around a point $p\in M$ satisfying the following assumptions: 
\begin{itemize}
\item $J_p$ takes its standard form with respect to $(z^1,\ldots,z^{2n})$;

\vspace{0.1cm}
\item $J_{i,k}^{\bar s}(p)=J_{i,\bar k}^{\bar s}(p)=J_{\bar j,k}^r(p)=J_{\bar j, \bar k}^r(p)=0$;

\vspace{0.1cm}
\item $
\hat g_{i\bar j}(p)=\delta_{ij}$, $
h_{i\bar j}(p)=\lambda_i\delta_{ij}$, $
\lambda_{2a-1}=\lambda_{2a}>0$.
\end{itemize}
Then
\[
\sum_{i,k=1}^{2n}
\frac{|h_{k\bar k,i}|^2}{\lambda_i\lambda_k}
\leq
2\sum_{i,j,r=1}^{2n}
\frac{|h_{i\bar j,r}|^2}{\lambda_i\lambda_j} \quad \mbox{ at }p\,.
\]
Consequently, if $\hat{g}$ is hyperk\"ahler
\[
\frac{|\partial\mathrm{tr}_{\hat g}h|_h^2}{\mathrm{tr}_{\hat g}h}
\leq
2\hat g^{r\bar s}h^{i\bar b}h^{a\bar j}
(\nabla^{\hat g}_{\bar s}h)_{a\bar b}
(\nabla^{\hat g}_r h)_{i\bar j}  \quad \mbox{ at }p\,.
\]
\end{prop}
\begin{proof}
Since the Obata connections is torsion-free, we can take $I$-holomorphic geodesic coordinates for it around $p$, then we may assume that these coordinates satisfy the assumptions of the proposition, see e.g. \cite[Remark 2.19]{DS23}. Hyperhermitianity of $h$ implies, at $p$,
\[
h_{i\bar j}
=J_i^{\bar s}J_{\bar j}^r h_{r \bar s},
\qquad
h_{i\bar j,k}
=J_i^{\bar s}J_{\bar j}^r h_{r \bar s,k}\,.
\]
Let $\Omega$ be the $(2,0)$-form associated with $h$. Since $h$ is HKT, we obtain at $p$
\[
0=(\partial \Omega)_{ij\ell}=\Omega_{ij,\ell}+\Omega_{j\ell,i}+\Omega_{\ell i,j}=(h_{i\bar q}J^{\bar q}_j)_{,\ell}+(h_{j\bar q}J^{\bar q}_\ell)_{,i}+(h_{\ell\bar q}J^{\bar q}_i)_{,j}=h_{i\bar q,\ell}J^{\bar q}_j+h_{j\bar q,i}J^{\bar q}_\ell+h_{\ell\bar q,j}J^{\bar q}_i\,.
\]
Now, multiplying by $J^\ell_{\bar k}$ and choosing $j=k$ we obtain, in particular
\[
h_{k\bar k,i}
=
h_{i\bar q,\ell}J^{\bar q}_kJ^{\ell}_{\bar k}+h_{i\bar k,k}\qquad \mbox{at }p\,.
\]

Write
\[
B_b=\{2b-1,2b\},
\qquad
\mu_b=\lambda_{2b-1}=\lambda_{2b},
\qquad b=1,\ldots,n\,,
\]
and set $k=s_b:=2b-1$. We then get
\[
|h_{s_b\bar s_b,i}|^2
\leq
2\left(
|h_{i\bar q,\ell}J^{\bar q}_{s_b}J^{\ell}_{\overline{s_b}}|^2+|h_{i\bar s_b,s_b}|^2\right)=2\sum_{j\in B_b}
|h_{i\bar j,j}|^2\qquad \mbox{at }p\,.
\]
Now, observe that, for $a\neq b$ after relabelling the indices
\[
\sum_{\substack{i\in B_a\\j\in B_b}}
|h_{j\bar i,i}|^2=
\sum_{\substack{i\in B_a\\j\in B_b}}
|h_{p\bar q,i}J^p_{\bar i}J^{\bar q}_j|^2=\sum_{\substack{i\in B_a\\j\in B_b}}
|h_{i\bar j, r}J^{ r}_{\bar i}|^2\qquad \mbox{at }p\,.
\]
We then infer
\[
\sum_{\substack{i\in B_a}}|h_{s_b\bar s_b,i}|^2+\sum_{\substack{j\in B_b}}|h_{s_a\bar s_a,j}|^2\leq 2\sum_{\substack{i\in B_a\\j\in B_b}}
(|h_{i\bar j,j}|^2+
|h_{j\bar i,i}|^2)
\leq
2\sum_{r=1}^{2n}\sum_{\substack{i\in B_a\\j\in B_b}}
|h_{i\bar j,r}|^2\qquad \mbox{at }p\,.
\]
Since the left-hand side is symmetric in $a$ and $b$ we deduce
\begin{equation}\label{eq:X}
\sum_{\substack{i\in B_a}}|h_{s_b\bar s_b,i}|^2+\sum_{\substack{j\in B_b}}|h_{s_a\bar s_a,j}|^2\leq \sum_{r=1}^{2n}\left( \sum_{\substack{i\in B_a\\j\in B_b}}
|h_{i\bar j,r}|^2+\sum_{\substack{i\in B_b\\j\in B_a}}
|h_{i\bar j,r}|^2\right)\qquad \mbox{at }p\,.
\end{equation}
On the other hand, for $a=b$ we have
\begin{equation}\label{eq:Y}
\sum_{i\in B_a}|h_{s_a\bar s_a,i}|^2\leq \sum_{r=1}^{2n}|h_{s_a\bar s_a,r}|^2=\frac{1}{2}\sum_{r=1}^{2n} \sum_{i,j\in B_a}|h_{i\bar j,r}|^2\qquad \mbox{at }p\,.
\end{equation}
We then deduce, using \eqref{eq:X} and \eqref{eq:Y}, at $p$
\[
\begin{aligned}
\sum_{i,k=1}^{2n}
\frac{|h_{k\bar k,i}|^2}{\lambda_i\lambda_k}
&=2\sum_{a=1}^n\sum_{i=1}^{2n}
\frac{|h_{s_a\bar s_a,i}|^2}{\lambda_i\mu_a}\\
&=
2\sum_{a=1}^n\frac{1}{\mu_a^2}\sum_{i\in B_a}|h_{s_a\bar s_a,i}|^2
+
2\sum_{a<b}\frac{1}{\mu_a\mu_b}\left(\sum_{\substack{i\in B_a}}|h_{s_b\bar s_b,i}|^2+\sum_{\substack{j\in B_b}}|h_{s_a\bar s_a,j}|^2\right)\\
&\leq
\sum_{a=1}^n\frac{1}{\mu_a^2}\sum_{r=1}^{2n} \sum_{i,j\in B_a}|h_{i\bar j,r}|^2
+
2\sum_{a<b}\frac{1}{\mu_a\mu_b}\sum_{r=1}^{2n}\left( \sum_{\substack{i\in B_a\\j\in B_b}}
|h_{i\bar j,r}|^2+\sum_{\substack{i\in B_b\\j\in B_a}}
|h_{i\bar j,r}|^2\right)\\
&\leq
2\sum_{a,b=1}^n\frac{1}{\mu_a\mu_b}\sum_{r=1}^{2n}\sum_{\substack{i\in B_a\\j\in B_b}}
|h_{i\bar j,r}|^2=
2\sum_{i,j,r=1}^{2n}
\frac{|h_{i\bar j,r}|^2}{\lambda_i\lambda_j}\,.
\end{aligned}
\]
This proves the first assertion. 

Finally, when $\hat g$ is hyperk\"ahler its first derivatives vanish at $p$, and weighted Cauchy--Schwarz yields
\[
\begin{aligned}
|\partial\mathrm{tr}_{\hat g}h|_h^2=
\sum_{i=1}^{2n}\frac1{\lambda_i}
\left|\sum_{k=1}^{2n} h_{k\bar k,i}\right|^2=
\sum_{i=1}^{2n}\frac1{\lambda_i}
\left|\sum_{k=1}^{2n} \sqrt{\lambda_k} \frac{h_{k\bar k,i}}{\sqrt \lambda_k}\right|^2\leq
\left(\sum_{k=1}^{2n}\lambda_k\right)
\sum_{i,k=1}^{2n}
\frac{|h_{k\bar k,i}|^2}{\lambda_i\lambda_k}\,.
\end{aligned}
\]
Now the second assertion follows from the first and the fact that at $p$, partial and covariant first derivatives with respect
to $\hat g$ agree, so that
\[
\sum_{i,j,r}
\frac{|h_{i\bar j,r}|^2}{\lambda_i\lambda_j}
=
\hat g^{r\bar s}h^{i\bar b}h^{a\bar j}
(\nabla^{\hat g}_{\bar s}h)_{a\bar b}
(\nabla^{\hat g}_r h)_{i\bar j}\,. \qedhere
\]
\end{proof}

Now we can prove \eqref{B-team}. 


Take normal coordinates for the hyperk\"ahler metric $\hat g$ on $U$ at $x_0$. So
$$
\hat g_{r\bar s,k}=J_{r,k}^{\bar s}=J_{r,\bar k}^{\bar s}=0   \quad \mbox{ at }x_0\,,
$$
see \cite[Remark 2.13]{DS23} or \cite[Proposition 3.1]{GV26b}. Then we directly compute at $x_0$   
$$
\begin{aligned}
\Delta_\phi \mathrm{tr}_{\hat g}g_\phi&= g_\phi^{i\bar j} \left((\hat g^{r\bar s})_{,i\bar j}g_{\phi \,r\bar s}+\hat g^{r\bar s} g_{\phi\, r\bar s, i \bar j}\right)\\
&=-g_\phi^{i\bar j}\hat g^{a\bar s}\hat g^{r\bar b}\hat g_{a\bar b,i\bar j}g_{\varphi\,r\bar s}+g_\phi^{i\bar j}\hat g^{r\bar s}\left( g_{r\bar s,i\bar j}+\frac12\phi_{r\bar s i\bar j}+\frac12J^{\bar a}_rJ^b_{\bar s}\phi_{b\bar a i\bar j} +\frac12J_{r,i\bar j}^{\bar a}J_{\bar s}^b\phi_{b\bar a}+\frac12J_{r}^{\bar a}J_{\bar s,i\bar j}^b\phi_{b\bar a}\right)\\
&=g_\phi^{i\bar j}\hat g^{a\bar s}\hat g^{r\bar b}\hat R_{a\bar bi\bar j}g_{\varphi\,r\bar s}
+g_\phi^{i\bar j}\hat g^{r\bar s}\left(g_{r\bar s,i\bar j}+\phi_{r\bar s i\bar j}\right) +\frac{1}{2}g_\phi^{i\bar j}\hat g^{r\bar s}\left(J_{r,i\bar j}^{\bar a}J_{\bar s}^b+J_{r}^{\bar a}J_{\bar s,i\bar j}^b\right) \phi_{b\bar a}\,,
\end{aligned}
$$
where $\hat R_{a\bar b i \bar j}$ are the components of the curvature of $\hat g$, and we used $\hat g^{r\bar s}J_r^{\bar a}J_{\bar s}^b=\hat g^{b\bar a}$ taking into account the $J$-Hermitianity of $\hat g$. Hence Lemmas 3.2 and 3.4 in \cite{GV26b} imply 
\begin{equation}\label{Delta}
\begin{aligned}
\Delta_\phi \mathrm{tr}_{\hat g}g_\phi=
g_\phi^{i\bar j}\hat g^{r\bar s}\left(g_{r\bar s,i\bar j}+\phi_{r\bar s i\bar j}\right)\,.
\end{aligned}
\end{equation} 
On the other hand, by differentiating twice the relation 
\[
\log\det(g_{\phi\,i\bar j})
=
2F+2b+\log\det(g_{i\bar j})\,, 
\]
we get
\[
\begin{split}
2\partial_{r}\partial_{\bar s} F+\partial_{r}\partial_{\bar s} \log\det(g_{i\bar j})&=
-g_{\varphi}^{i\bar b}g_{\varphi}^{a\bar j}g_{\varphi\,a\bar b,\bar s}g_{\varphi\,i\bar j,r}
\\
&\quad +g_{\varphi}^{i\bar j}\left(g_{i\bar j,r\bar s}
 + \frac12\varphi_{r\bar s i\bar j }+\frac12 J_{i}^{\bar b}J_{\bar j}^a\varphi_{r\bar s a\bar b}
+ \frac12J_{i}^{\bar b}J_{\bar j,r\bar s}^a\varphi_{a\bar b}
+ \frac12J_{i,r\bar s}^{\bar b}J_{\bar j}^a\varphi_{a\bar b}\right)\,.
\end{split}
\]
By tracing with respect to $\hat g$ and applying \cite[Lemma 3.4]{GV26b} the terms involving the second derivatives of $J$ vanish and we infer 
\[
2\Delta_{\hat g} F+\Delta_{\hat g}\log \det(g_{i\bar j})=-\hat g^{r   \bar s }g_{\varphi}^{i\bar b}g_{\varphi}^{a\bar j}g_{\varphi\,a\bar b,\bar s}g_{\varphi\,i\bar j,r}
+\hat g^{r   \bar s }g_{\varphi}^{i\bar j}(g_{i\bar j,r\bar s}
+ \varphi_{r\bar s i\bar j })\,.
\]
Consequently, from \eqref{Delta}, we obtain
\[
\Delta_\phi\mathrm{tr}_{\hat g}g_\phi
=\hat g^{r\bar s}g_\phi^{i\bar b}g_\phi^{a\bar j}
g_{\phi\,a\bar b,\bar s}g_{\phi\,i\bar j,r}+
g_\phi^{i\bar j}\hat g^{r\bar s}
(g_{r\bar s,i\bar j}-g_{i\bar j,r\bar s})
+2\Delta_{\hat g}F
+\Delta_{\hat g}\log\det(g_{i\bar j})\,.
\]
Thus
\[
\Delta_\phi\mathrm{tr}_{\hat g}g_\phi
\geq \hat g^{r\bar s}g_\phi^{i\bar b}g_\phi^{a\bar j}
g_{\phi\,a\bar b,\bar s}g_{\phi\,i\bar j,r}-C_1(1+\mathrm{tr}_{g_\phi}g)\,.
\]
for a positive constant $C_1$ independent of $\varphi$. Combining this with the estimate of Proposition \ref{Alesker} and dividing by
$\mathrm{tr}_{\hat g}g_\phi>0$, we obtain \eqref{B-team}.

\medskip 
Using \eqref{B-team} in \eqref{C-team} we deduce
\[
0\geq\frac12 \left| \partial\log\mathrm{tr}_{\hat g}g_\phi
\right|_{g_\phi}^{2}- \frac{C_1(1+\mathrm{tr}_{g_\phi}g)}{\mathrm{tr}_{\hat g}g_\phi}+2|\partial f|_{g_\phi}^2-2
\mathrm{Re}\langle\partial f,\partial\log \mathrm{tr}_{\hat g}g_\phi\rangle_{g_\phi}
+\frac{\mathrm{tr}_{g_\phi}g-2n}{A+\phi} \quad \mbox{ at }x_0\,.
\]
Now we recognise and drop the square
\[
\frac12 \left| \partial\log\mathrm{tr}_{\hat g}g_\phi
\right|_{g_\phi}^{2}+2|\partial f|_{g_\phi}^2-2
\mathrm{Re}\langle\partial f,\partial\log \mathrm{tr}_{\hat g}g_\phi\rangle_{g_\phi}=\frac{1}{2}\left| \partial\log\mathrm{tr}_{\hat g}g_\phi
-2\partial f\right|_{g_\phi}^{2} \geq 0\,.
\]
Therefore, we have
\[
0\geq- \frac{C_1(1+\mathrm{tr}_{g_\phi}g)}{\mathrm{tr}_{\hat g}g_\phi}+\frac{\mathrm{tr}_{g_\phi}g-2n}{A+\phi} \quad \mbox{ at }x_0\,.
\]
We consider two cases. If $\mathrm{tr}_{\hat g}g_\phi < 2C_1(A+\phi)$ at $x_0$ we obtain a uniform upper bound for $\mathrm{tr}_{\tilde g}g_\phi$ and we are done. Therefore, we may assume $\mathrm{tr}_{\hat g}g_\phi \geq  2C_1(A+\phi)$ and thus
\[
0\geq \left(\frac{1}{A+\phi}-\frac{C_1}{\mathrm{tr}_{\hat g}g_\phi}\right)\mathrm{tr}_{g_\phi}g-\frac{C_1}{\mathrm{tr}_{\hat g}g_\phi}-\frac{2n}{A+\phi}\geq \frac{1}{2(A+\phi)}\mathrm{tr}_{g_\phi}g -C_2\quad \mbox{ at }x_0\,,
\]
where $C_2$ is a positive constant depending on $C_1$ and $\|\phi\|_{C^0}$. Now we deduce
\[
\mathrm{tr}_{g_\phi}g\leq C_3\quad \mbox{ at }x_0\,, 
\]
where $C_3>0$ depends on $C_2$ and $\|\phi\|_{C^0}$. Also
$$
{\rm tr}_{ g} g_\phi \leq \frac{1}{(2n-1)!} ({\rm tr}_{g_\phi } g)^{2n-1}
\frac{\omega^{2n}_\phi }{ \omega^{2n}}\leq \frac{1}{(2n-1)!} C^{2n-1}_3 \mathrm{e}^{2F+2b}
$$
and the fact that there exists a constant $C_4>0$ such that $C_4^{-1}g\leq \tilde g\leq C_4 g$, imply
$$
{\rm tr}_{\tilde g} g_\varphi\leq C_5 \quad \mbox{ at }x_0\,,
$$
where $C_5>0$ depends on $C_3$, $C_4$, $F$ and $b$ (which is uniformly bounded by $F$). It follows
$$
\log \mathrm{tr}_{\tilde g}g_\phi=Q+\log (A+\phi)\leq Q(x_0)+\log (A+\phi)=\log \mathrm{tr}_{\tilde g}g_\phi(x_0) -\log (A+\phi(x_0))+\log (A+\phi)\leq C_6\,,
$$
where $C_6>0$ depends on $C_5$ and $\|\phi\|_{C^0}$.  

We thus have proved the desired upper bound on $\mathrm{tr}_{\tilde g}g_\phi$, concluding the proof of Theorem \ref{main1}.

\begin{rmk}
{\em Proposition \ref{Alesker} is the natural generalization of Alesker's result \cite[Proposition 3.1]{Ale13} to our setting. }
\end{rmk}

\begin{rmk}\label{rem}
{\em Our result does not naturally extend to the \lq\lq general\rq\rq  \,case without the assumption of the existence of a background LCHK metric. For instance the crucial inequality \eqref{B-team} does not hold in general if we replace the hyperk\"ahler metric with an HKT metric. A counterexample is the following: 

Work on a neighbourhood of the origin in
$\mathbb H^2\simeq\mathbb C^4$ and set 
\[
a(z)=|z_1|^2+|z_2|^2\,,
\qquad
b(z)=|z_3|^2+|z_4|^2\,.
\]
Consider the background metric
\[
\omega=\omega_0+\left(i \partial \bar \partial(a(z)b(z))\right)^-\,,
\]
where $\omega_0$ is the standard metric on $\C^4$. In our conventions
$$
Jd\bar z_1=dz_2,\qquad Jd\bar z_2=-dz_1,\qquad Jd\bar z_3=dz_4,\qquad Jd\bar z_4=-dz_3\,,
$$
and so, the matrix associated to $\omega$ is 
\[
\mathrm{Id}+ \begin{pmatrix} b(z)&0& \dfrac{\bar z_1z_3+\bar z_4z_2}{2}& \dfrac{\bar z_1z_4-\bar z_3z_2}{2}\\[3mm] 0&b(z)& \dfrac{\bar z_2z_3-\bar z_4z_1}{2}& \dfrac{\bar z_2z_4+\bar z_3z_1}{2}\\[3mm] \dfrac{\bar z_3z_1+\bar z_2z_4}{2}& \dfrac{\bar z_3z_2-\bar z_1z_4}{2}& a(z)&0\\[3mm] \dfrac{\bar z_4z_1-\bar z_2z_3}{2}& \dfrac{\bar z_4z_2+\bar z_1z_3}{2}& 0&a(z)\end{pmatrix}
\]
which is positive definite on a sufficiently small neighbourhood of the origin. Clearly, at the origin $g_{i\bar j}=\delta_{i\bar j} $ and $ g_{i\bar j,k}=0 $.

For $t>1$, define
\[
\phi_t=(t-1)a(z)+(t^{-1}-1)b(z)-a(z)b(z).
\]
Then $ \omega_{\phi_t} := \omega+(i\partial \bar \partial \phi_t)^- $ has the matrix expression $
\mathrm{diag}(t,t,t^{-1},t^{-1})$, which is positive definite. In particular, since $\det(g_{\phi_t})=1$ is constant in $t$, writing $F=-\frac12\log\det(g_{i\bar j})$, the associated $(2,0)$ form $\Omega_{\phi_t}$ solves
\[
\Omega_{\phi_t}^{2}=\mathrm{e}^F\Omega^2\,, \qquad \text{for every }t>1\,,
\]
where $\Omega$ is the $(2,0)$ form associated to $g$. At the origin,
\[
\mathrm{tr}_g g_{\phi_t}
=\mathrm{tr}_{g_{\phi_t}}g
=2(t+t^{-1})\,, \qquad \partial\mathrm{tr}_g g_{\phi_t}=0\,, \qquad \Delta_{\phi_t}\mathrm{tr}_g g_{\phi_t} = -4(t^2+t^{-2})\,.
\]
Consequently, at the origin
\[
\frac{\Delta_{\phi_t}\mathrm{tr}_g g_{\phi_t}}{\mathrm{tr}_g g_{\phi_t}} =-2\frac{t^2+t^{-2}}{t+t^{-1}} \longrightarrow-\infty \qquad\text{as }t\to\infty\,,
\]
whereas
\[
\frac{1+\mathrm{tr}_{g_{\phi_t}}g}{\mathrm{tr}_g g_{\phi_t}}
= 1+\frac{1}{2(t+t^{-1})}
\]
remains uniformly bounded, thus violating the inequality \eqref{B-team}.
}
\end{rmk}

\section{Proof of Theorem \ref{main2}}
Let $T_{\max}$ denote the maximal existence time of the flow \eqref{flow}. By definition $T_{\max}\leq T$, where $T$ is the one appearing in Conjecture \ref{conj}. Assume by contradiction that $T_{\max}<T$. Thanks to \cite[Proposition 5.2]{BGV23} or \cite[Theorem 1.4]{BS26} it is enough to provide an upper bound for $\mathrm{tr}_{g_0}g(t)$ on $[0,T_{\max})$.

Choose $T_{\max}<\tau<T$, then by definition there exists a smooth function $\psi$ such that
\[
\chi_\tau
=\omega_0-\tau\mathrm{Ric}(\omega_0)^{-}
 +(i\partial\bar\partial\psi)^->0\,.
\]
The family
\[
\chi_t
=\left(1-\frac t\tau\right)\omega_0
 +\frac t\tau\chi_\tau
\]
is positive and uniformly equivalent to $\omega_0$ on $[0,T_{\max}]$.
Write
\[
\omega_t=\chi_t+(i\partial\bar\partial \phi(t))^-\,,
\]
where the potential $\phi$ satisfies
\[
\dot \phi
=\log\frac{\omega_t^{2n}}{\omega_0^{2n}}
 -\frac{\psi}{\tau}\,, \qquad \phi(0)=0\,.
\]
Then $\omega_t $ is the solution of the flow \eqref{flow}. By Lemmas 5.3 and 5.4 in \cite{BGV23} we deduce that $\phi$ and $\dot \phi$ are uniformly bounded from above and from below.

Let $A=1+\|\phi\|_{C^0(M\times[0,T_{\max}))}$, and consider the test function
\[
Q=\log\mathrm{tr}_{\tilde g}g(t)-\log(A+\phi).
\]
Fix $t_1<T_{\max}$ and let $(x_0,t_0)$ be a maximum point of
$Q$ on $M\times[0,t_1]$.
Now, either $t_0=0$ in which case we have
\[
\log\mathrm{tr}_{\tilde g}g(t)-\log(A+\phi)\leq \log\mathrm{tr}_{\tilde g}g_0(x_0)-\log(A)
\]
and the estimate follows, or
\[
\frac{\partial \mathrm{tr}_{\tilde g}g(t)}{\mathrm{tr}_{\tilde g}g(t)}=\frac{\partial \phi}{A+\phi}\quad \mbox{ at }(x_0,t_0)\,,
\]
and so
\begin{equation}\label{S-team}
0\geq (\Delta_{g(t)}-\partial_t)Q=\frac{\Delta_{g(t)}\mathrm{tr}_{\tilde g}g(t)}
{\mathrm{tr}_{\tilde g}g(t)}
-\frac{\Delta_{g(t)} \phi}{A+\phi}-\frac{\partial_t \mathrm{tr}_{\tilde g} g(t)}{\mathrm{tr}_{\tilde g}g(t)}+\frac{\dot \phi}{A+\phi}\quad \mbox{ at }(x_0,t_0)\,,
\end{equation}
Choose the Gauduchon representative $\tilde g$ of the auxiliary LCHK conformal class. Locally write $\tilde g={\rm e}^f\hat g,$ where $\hat g$ is hyperk\"ahler. Then, the inequality \eqref{S-team} rewrites, at $(x_0,t_0)$
\[
0\geq\frac{\Delta_{g(t)}\mathrm{tr}_{\hat g}g(t)}
{\mathrm{tr}_{\hat g}g(t)}-\Delta_{g(t)}f+|\partial f|^2_{g(t)}-\frac{2}{\mathrm{tr}_{\hat g}g(t)}\Re\langle \partial f, \partial \mathrm{tr}_{\hat g}g(t)\rangle_{g(t)}
+\frac{\mathrm{tr}_{g(t)}g_{\chi_t}-2n}{A+\phi}-\frac{\partial_t \mathrm{tr}_{\hat g} g(t)}{\mathrm{tr}_{\hat g}g(t)}+\frac{\dot \phi}{A+\phi}\,,
\]
where we also used $\Delta_{g(t)} \phi=2n-\mathrm{tr}_{g(t)}g_{\chi_t}.$

Using the cancellations of Lemmas 3.2 and 3.4 in \cite{GV26b} we have
\[
\Delta_{g(t)} \mathrm{tr}_{\hat g}g(t)=g(t)^{i\bar j}\hat g^{r\bar s}((g_{\chi_ t})_{r \bar s,i\bar j}+\phi_{r\bar s i\bar j})\,.
\]
On the other hand, 
\[
\partial_t \mathrm{tr}_{\hat g}g(t)=\hat g^{r \bar s}(\log \det g(t)_{i\bar j})_{r\bar s}=\hat g^{r \bar s}\left(g(t)^{i\bar j}((g_{\chi_ t})_{i\bar j,r \bar s}+\phi_{r\bar s i\bar j})-g(t)^{i\bar b}g(t)^{a\bar j}g(t)_{a\bar b,\bar s}g(t)_{i\bar j,r}\right)\,.
\]
We therefore obtain
\[
(\Delta_{g(t)}-\partial_t)\mathrm{tr}_{\hat g}g(t)
\geq \hat g^{r \bar s}g(t)^{i\bar b}g(t)^{a\bar j}g(t)_{a\bar b,\bar s}g(t)_{i\bar j,r}-C_1(1+\mathrm{tr}_{g(t)}g_{\chi_t})
\]
for a positive constant $C_1$. Applying this to \eqref{S-team}, together with Lemma \ref{Gauduchon} and Proposition \ref{Alesker}, gives
\begin{equation}\label{T-team}
0\geq
\left(\frac1{A+\phi}-\frac{C_1}{\mathrm{tr}_{\hat g}g(t)}\right)\mathrm{tr}_{g(t)}g_{\chi_t}-C_2-\frac{C_1}{\mathrm{tr}_{\hat g}g(t)}\quad \mbox{ at }(x_0,t_0)\,,
\end{equation}
where we also dropped the square
\[
\frac12 \left| \partial\log\mathrm{tr}_{\hat g}g(t)
\right|_{g(t)}^{2}+2|\partial f|_{g(t)}^2-2
\mathrm{Re}\langle\partial f,\partial\log \mathrm{tr}_{\hat g}g(t)\rangle_{g(t)}=\frac{1}{2}\left| \partial\log\mathrm{tr}_{\hat g}g(t)
-2\partial f\right|_{g(t)}^{2} \,,
\]
and $C_2>1$ is a constant depending on $\dot \phi$. Since $1\leq A+\phi\leq2A$, either $\mathrm{tr}_{\hat g}g(t)$ is already bounded by $2C_1(A+\phi)$, or we deduce $\mathrm{tr}_{g(t)}g_{\chi_t}\leq C_3$ at $(x_0,t_0)$ from \eqref{T-team}.
In the latter case, the uniform equivalence of $g_{\chi_t}$ and $\hat g$ as well as the uniform determinant bound (which is a consequence of the boundedness of $\dot \phi$), also gives a uniform upper bound for $\mathrm{tr}_{\hat g}g(t)$.

The fact that we have been working at a maximum point for $Q$ and the bound on $\phi$ imply
\[
\mathrm{tr}_{\tilde g}g(t)\leq C
\quad\text{on }M\times[0,t_1].
\]
Since $t_1<T_{\max}$ is arbitrary and $\tilde g$ and $g_0$ are uniformly equivalent, we get
\[
C^{-1}g_0\leq g(t)\leq Cg_0
\quad\text{on }M\times[0,T_{\max}).
\]

Since $\chi_t$ is a smooth family, uniformly positive, with uniformly bounded derivatives, we can apply \cite[Theorem 3.3]{BGV23}. The parabolic $C^{2,\alpha}$ estimate for the potential, followed by Schauder estimates, now yields uniform higher-order bounds away from the initial time. Consequently the solution extends smoothly beyond $T_{\max}$, contradicting maximality. Thus $T_{\max}=T$.

\section{Proof of Theorem \ref{main3}}
In this section we focus on the proof of Theorem \ref{main3}.

\medskip 
Let $H_\lambda$ be a quaternionic Hopf surface. For $f\colon H_\lambda\to \R$, we have 
$$
\mathrm{Ric}_I({\rm e}^f\, g_H)^+=\big(1-\tfrac{|q|^2}{4}\Delta_{\rm Eu}f\big) g_H\,,
$$
where $\Delta_{\rm Eu}f$ is the Euclidean Laplacian applied to the lift of $f$ on $\mathbb{H}\backslash \{0\}$. The Lee vector field $V=-\frac{1}{2}\theta^\sharp$ is the projection of $\frac12 \sum_{k=0}^3x^k\partial_{x^k}$ onto $H_\lambda$. Hence if we assume $f([q])=u(\log |q|)$ for a smooth function $u$ which is  $(\log\lambda)$-periodic, we have  
$$
{\rm Ric}_{I}({\rm e}^{f}\,g_H)^{+}=\left(1-\frac14u''-\frac12 u'\right) g_H\quad \mbox{ and }\quad \mathcal{L}_{V}({\rm e}^{f}\,g_H)=\frac12 u'{\rm e}^{f}\, g_H\,,
$$
where the derivatives of $u$ are computed at $\log |q|$. Hence the soliton equation 
$$
{\rm Ric}_I({\rm e}^{f}\,g_H)^{+}={\rm e}^{f}\,g_H-\mathcal{L}_{V}({\rm e}^{f}\, g_H)
$$ 
reads in terms of $u$ as the ODE
\begin{equation}\label{ODE}
u''+2(1-{\rm e}^u)u'+4({\rm e}^u-1)=0\,.
\end{equation}
It remains to show that \eqref{ODE} has a smooth $\mu$-periodic solution for every $\mu>\pi$. In this way it induces a soliton on $H_\lambda$ for $\lambda={\rm e}^\mu$. 
The function 
\[
H(u,v)={\rm e}^u-u-1-\frac{v}{2}
-\log\left(1-\frac{v}{2}\right)
\]
is a first integral of the ODE. Note that 
\begin{itemize}
\item $H$ is proper and surjective as map $U:=\{{(u,v)\in \R^2\,\,:\,\,v<2} \}\to [0,\infty)$;

\vspace{0.1cm}
\item since $\nabla H=\left({\rm e}^u-1,\frac{v}{2(2-v)} \right)$, the only critical point of $H$ on $U$ is the origin $(0,0)$;

\vspace{0.1cm}
\item $(0,0)$ is the only zero of $H$ in $U$. 

\end{itemize} 
It follows that solutions to equation \eqref{ODE} with initial data $(0,0)\neq (u(0),\dot u(0))\in U$ are periodic. Let $u$ be the solution to \eqref{ODE} such that $u(0)=a>0$, $\dot u(0)=0$, and denote by $\mu(a)$ the period of $u$ (first return time to the transversal half-line $\{v=0,\,u>0\}$). Then $\mu$ is continuous, $\mu(a)\to\pi$ as $a\to0$ (the rescaling $u=aw$ reduces \eqref{ODE} to $\ddot w+4w=0$ at $a=0$), and $\mu(a)\ge a/2$ since $\dot u<2$. Hence the image of $\mu$ contains $(\pi,\infty)$ and the first part of the claim follows.  

\medskip

Now we focus on the second part of the statement. We may assume without loss of generality $c=1$ and we write 
\[
g_t={\rm e}^{f}g_H
\]
for a smooth function $f\colon H_\lambda \times [0,\infty)\to \R$. The assumption $c=1$ translates into
\begin{equation}\label{eq:c=1}
\int_{H_\lambda}{\mathrm e}^{f(\cdot,0)}\,d\mu_{g_H}
=\mathrm{Vol}(g_H)\,.
\end{equation}
We recall that $H_\lambda$ is diffeomorphic to $S^1\times S^3$ via the map induced by 
$$
\R\times S^3\to H_{\lambda}\,,\quad (s,q)\mapsto {\rm e}^sq\,,
$$
where $S^3$ is regarded as $\{q\in \mathbb{H}\,\,:\,\, |q|=1\}$ and the circle $S^1$ as $\R/ \log\lambda \mathbb Z$. 
Under this identification the vector field $V$ is tangent to $S^1$ and $V$-invariant functions on $H_\lambda$ can be regarded as smooth functions on $S^3$. 

Assume $\mathcal{L}_V(g_0)=0$; then 
$\mathcal{L}_V(g_t)=0$ for every $t\in [0,\infty)$ and $f$ descends to a function on $S^3$. The normalized HKT flow can then be rewritten as 
\begin{equation}\label{1}
\partial_tf=\frac14{\rm e}^{-f}\Delta_{S^3}f+1-{\rm e}^{-f}\,.
\end{equation}
Now, \eqref{eq:c=1} yields $ \int_{S^3} \mathrm{e}^{f(\cdot,0)}\, d\mu_{S^3}=\mathrm{Vol}(S^3) $ and it is not difficult to see that
\[
\int_{S^3} \mathrm{e}^{f}\, d\mu_{S^3}=\mathrm{Vol}(S^3) \qquad \text{for every }t\in [0,+\infty)\,.
\]
The uniform equivalence assumption 1. in the statement is equivalent to a uniform $C^0$ bound 
\[
|f|\leq C_0 \,.
\]
Hence this assumption implies that equation \eqref{1} is uniformly parabolic and 
standard parabolic regularity implies for every $k\geq0$ a uniform bound of the form 
\begin{equation}
\sup_{t\geq1}\|f(\cdot,t)\|_{C^k(S^3)}\leq C_k\,.
\label{2}
\end{equation}
Consider now the functional $\mathcal F\colon C^2(S^3)\to \R$ defined as 
\[
\mathcal F(u)
=
\int_{S^3} \left( \frac18 |\nabla u|^2+u-{\rm e}^u \right)\,d\mu_{S^3}\, .
\]
Its $L^2$-gradient is
\[
\nabla\mathcal F(u)
=
-\frac14\Delta_{S^3} u+1-{\rm e}^u\,.
\]
Using \eqref{1}, we obtain
\[
\nabla\mathcal F(f)=-{\rm e}^f \partial_tf
\]
and 
\[
\partial_t\mathcal F(f)
=
\int_{S^3} \nabla \mathcal F(f)\partial_t f\,d\mu_{S^3}\\
=
-\int_{S^3} {\rm e}^f (\partial_tf)^2\,d\mu_{S^3}
\leq0\,.
\]
Hence for $\tau>0$ we have 
$$
\int_{0}^\tau\int_{S^3} {\rm e}^f (\partial_tf)^2\,d\mu_{S^3}\,dt=\mathcal{F}(f(\cdot,0))-\mathcal{F}(f(\cdot,\tau))\,.
$$
Since $f$ is uniformly bounded, $\mathcal F$ is bounded from below along the flow. Therefore we obtain
\[
\int_0^\infty\int_{S^3} {\rm e}^f (\partial_tf)^2\,d\mu_{S^3}\,dt<\infty
\]
and there exists a sequence $t_j\to\infty$ such that $
\|\partial_tf(\cdot,t_j)\|_{L^2}\to 0. $ By the precompactness provided by \eqref{2}, after passing to a subsequence,
\[
f(\cdot,t_j)\longrightarrow \bar f
\quad\text{ in }C^\infty(S^3)\,.
\]
Passing to the limit in \eqref{1}, we obtain
\[
\frac14\Delta_{S^3} \bar f+{\rm e}^{\bar f}-1=0
\]
and the metric 
\[
\bar g={\rm e}^{\bar f}g_H
\]
is HKT--Einstein. 
It remains to prove convergence of the whole trajectory. In a $C^{2,\alpha}$-neighborhood of $\bar f$, the functional $\mathcal F$ is analytic and its Hessian at $\bar f$ as map $C^{2,\alpha}(S^3)\to C^{0,\alpha}(S^3)$ is 
\[
\nabla^2\mathcal F_{\bar f}(u)
=
-\frac14\Delta_{S^3} u-{\rm e}^{\bar f}u\,,
\]
which is an elliptic self-adjoint Fredholm operator.

Hence the \L ojasiewicz--Simon inequality applies: there exist $\vartheta\in(0,\frac12]$, $C>0$, and a $C^{2,\alpha}$-neighbourhood of $\bar f$ such that
\[
|\mathcal F(u)-\mathcal F(\bar f)|^{1-\vartheta}
\leq
C\|\nabla\mathcal F(u)\|_{L^2}\,.
\]
Since $\mathcal F(f(t))$ is non-increasing and converges to $\mathcal F(\bar f)$, the Łojasiewicz--Simon inequality, together with $\nabla\mathcal F(f)=-e^f\partial_tf$ and $|f|\le C_0$ give $-\frac{d}{dt}\big(\mathcal F(f)-\mathcal F(\bar f)\big)^{\vartheta}\ge c\|\partial_tf\|_{L^2}$ as long as $f(t)$ stays $C^{2,\alpha}$-close to $\bar f$. By \eqref{2} and interpolation, a standard continuity argument  shows that this holds for all $t\ge t_j$, $j$ large, so $\int_{t_j}^\infty\|\partial_tf\|_{L^2}\,dt<\infty$ and $f(\cdot,t)\to\bar f$ in $L^2$.

Combining this with the uniform higher order estimates \eqref{2}, we conclude that
\[
f(\cdot,t)\longrightarrow \bar f
\qquad\text{in }C^\infty(S^3)
\]
and the claim follows.

\medskip
\medskip
Finally, assume that $f(\cdot,0)$ is also constant along $IV$. Hence $f(\cdot,t)$ is constant along $IV$ for every $t\in [0,\infty)$ and descends to the base of the Hopf fibration
\[
\begin{tikzcd}
S^1 \arrow[r, hook] & S^3 \arrow[d, "\pi"] \\
                    & S^2
\end{tikzcd}
\]
Since $\pi\colon S^3(1)\to S^2(\tfrac12)$ is a Riemannian submersion with totally geodesic fibers,
$\tfrac14\Delta_{S^3}$ acts on basic functions as the Laplacian $\Delta_{S^2}$ of the unit sphere $S^2$, and
\eqref{1} writes as
\[
e^{f}\partial_tf=\Delta_{S^2} f+e^{f}-1\quad\text{on }S^2 \,.
\]
Since $\mathrm{tr}_{g_H}g_t=2e^{f}$, the assumption $2\,\mathrm{Vol}(g_H)=\int_{H_\lambda}\mathrm{tr}_{g_H}g_0\,d\mu_{g_H}$
reads $\int_{S^2}e^{f(\cdot,0)}d\mu_{S^2}=4\pi$, and it is preserved since
$\frac{d}{dt}\int_{S^2}(e^f-1)\,d\mu_{S^2}=\int_{S^2}(e^f-1)\,d\mu_{S^2}$. Hence the equation can be written as
\[
\partial_t e^{f}=\Delta_{S^2}f-1+4\pi\,\frac{e^{f}}{\int_{S^2}e^{f}\,d\mu_{S^2}}\,,
\]
which is the mean field type flow of \cite{Cas15} with $Q\equiv1$ and $\rho=4\pi<8\pi$. By
\cite[Theorem~0.2(i)]{Cas15}, $f(\cdot,t)$ converges in $H^2(S^2)$ as $t\to\infty$. Since
$H^2(S^2)\hookrightarrow C^0(S^2)$, we get $|f|\le C_0$, i.e.\ condition 1. in Theorem \ref{main3} holds, and the claim follows
from the first part of the proof. \qed

\begin{rmk}
{\em The standard metric $g_H$ is, up to rescaling, the unique
HKT--Einstein metric on $H_\lambda$ invariant under
both $V$ and $IV$. Indeed, write such a metric as
${\rm e}^f g_H$. The invariance implies that $f$ descends
to the round unit sphere $S^2$, and the HKT--Einstein equation
${\rm Ric}_I({\rm e}^f g_H)^+=c\,{\rm e}^f g_H$ becomes
\[
\Delta_{S^2}f+c\,{\rm e}^f=1.
\]
Integration shows that $c>0$, so, after rescaling the metric,
we may assume $c=1$. Then
$\int_{S^2}{\rm e}^f\,dA=4\pi$, and
\cite[Theorem 1.1]{Lin00}, applied with $\alpha=4\pi$,
$\gamma=0$, and $\phi=(f-\bar f)/(4\pi)$, where $\bar f$
denotes the average of $f$, implies that $f$ is constant.
The equation then yields $f\equiv0$.}
\end{rmk}






\medskip

\bibliographystyle{siam}
\bibliography{bibliography}

\end{document}